\documentclass{amsart}
\usepackage[utf8]{inputenc}
\usepackage{lineno}

\usepackage[a4paper, top=3cm, left=2.5cm, right=2.5cm, bottom=3cm]{geometry}
\usepackage{amssymb,comment}
\usepackage{graphicx,color}
\usepackage{enumerate}

\usepackage[hidelinks]{hyperref}

    \newcommand{\menos}{\smallsetminus}

    \newcommand{\M}{\mathbb{M}}
    
    \newcommand{\Z}{\mathbb{Z}}
    \newcommand{\Q}{\mathbb{Q}}

    \newcommand{\imp}{{\ \mbox{$\Rightarrow$} \ }}

    \newcommand{\sii}{{\ \mbox{$\Leftrightarrow$} \ }}

    \newcommand{\Fbb}{\mathbb{F}}
    \newcommand{\Lbb}{\mathbb{L}}
    \newcommand{\Ubb}{\mathbb{U}}

    \newcommand{\GCD}{\mathrm{gcd}}
    \newcommand{\LCM}{\mathrm{lcm}}
    \newcommand{\lege}[2]{\left(\frac{#1}{#2}\right)}

\title{Some notes about power residues modulo prime}

\author{Yuki Kiriu}
\address{Shizuoka Salesio High School, Nakanogo 3-2-1, Shimizu-ku, Shizuoka-shi, Japan 424-8624}
\email{\href{mailto:uki_sal@yahoo.co.jp}{uki\_sal@yahoo.co.jp}}

\author{Diego A. Mejía}
\address{Creative Science Course (Mathematics), Faculty of Science, Shizuoka University, Ohya 836, Suruga-ku, Shizuoka-shi, Japan 422-8529.}
\email{\href{mailto:diego.mejia@shizuoka.ac.jp}{diego.mejia@shizuoka.ac.jp}}
\urladdr{\url{http://www.researchgate.com/profile/Diego\_Mejia2}}

\thanks{This work was supported by:
Future Scientists School at Shizuoka University, Global Science Campus supported by the Japan Science and Technology Agency (both authors);
Grant-in-Aid for Early Career Scientists 18K13448, Japan Society for the Promotion of Science (second author).}



\date{\today}

\begin{document}

\makeatletter
\def\@roman#1{\romannumeral #1}
\makeatother

\newcounter{enuAlph}
\renewcommand{\theenuAlph}{\Alph{enuAlph}}

\numberwithin{equation}{section}
\renewcommand{\theequation}{\thesection.\arabic{equation}}

\theoremstyle{plain}
  \newtheorem{theorem}[equation]{Theorem}
  \newtheorem{corollary}[equation]{Corollary}
  \newtheorem{lemma}[equation]{Lemma}
  \newtheorem{mainlemma}[equation]{Main Lemma}
  \newtheorem{prop}[equation]{Proposition}
  \newtheorem{clm}[equation]{Claim}
  \newtheorem{exer}[equation]{Exercise}
  \newtheorem{question}[equation]{Question}
  \newtheorem{problem}[equation]{Problem}
  \newtheorem{conjecture}[equation]{Conjecture}
  \newtheorem*{thm}{Theorem}
  \newtheorem{teorema}[enuAlph]{Theorem}
  \newtheorem*{corolario}{Corollary}
\theoremstyle{definition}
  \newtheorem{definition}[equation]{Definition}
  \newtheorem{example}[equation]{Example}
  \newtheorem{remark}[equation]{Remark}
  \newtheorem{notation}[equation]{Notation}
  \newtheorem{context}[equation]{Context}

  \newtheorem*{defi}{Definition}
  \newtheorem*{acknowledgements}{Acknowledgements}

\begin{abstract}
   Let $q$ be a prime. We classify the odd primes $p\neq q$ such that the equation $x^2\equiv q\pmod{p}$ has a solution, concretely, we find a subgroup $\mathbb{L}_{4q}$ of the multiplicative group $\mathbb{U}_{4q}$ of integers relatively prime with $4q$ (modulo $4q$) such that $x^2\equiv q\pmod{p}$ has a solution iff $p\equiv c\pmod{4q}$ for some $c\in\mathbb{L}_{4q}$. Moreover, $\mathbb{L}_{4q}$ is the only subgroup of $\mathbb{U}_{4q}$ of half order containing $-1$.

   Considering the ring $\mathbb{Z}[\sqrt{2}]$, for any odd prime $p$ it is known that the equation $x^2\equiv 2\pmod{p}$ has a solution iff the equation $x^2-2y^2=p$ has a solution in the integers. We ask whether this can be extended in the context of $\mathbb{Z}[\sqrt[n]{2}]$ with $n\geq 2$, namely: for any prime $p\equiv 1\pmod{n}$, is it true that $x^n\equiv 2\pmod{p}$ has a solution iff the equation $D^2_n(x_0,\ldots,x_{n-1})=p$ has a solution in the integers? Here $D^2_n(\bar{x})$ represents the norm of the field extension $\mathbb{Q}(\sqrt[n]{2})$ of $\mathbb{Q}$. We solve some weak versions of this problem, where equality with $p$ is replaced by $0\pmod{p}$ (divisible by $p$), and the ``norm" $D^r_n(\bar{x})$ is considered for any $r\in\mathbb{Z}$ in the place of $2$.
\end{abstract}

\maketitle

\newcommand{\adj}{\mathrm{adj}}

\section{Introduction}\label{sec:intro}

In this work, we prove several properties and present problems related with  quadratic residues and its generalization to $n$-th power residues modulo prime, all in the framework of elementary number theory.

Before entering into the subject, we first fix some basic notation.

\begin{notation}\label{notat}
In the following, $m>1$ is an integer and $q$ is a prime.
\begin{enumerate}[(1)]
\item $\Fbb_q$ denotes the field of integers modulo $q$, which is the prime field of order $q$, and $\Fbb^\times_q$ denotes its associated multiplicative group.

\item More generally, $\Ubb_m$ denotes the multiplicative group of integers modulo $m$ that are relatively prime with $m$. Note that $\Ubb_q=\Fbb^\times_q$.

\item Let $G$ be a group with identity element $1_G$. For any $r\in G$, the \emph{order of $r$ in $G$}, which we denote by $O_G(r)$, is the smallest positive integer $n$ satisfying $r^n=1_G$ in case it exists, otherwise $O_G(r)$ is \emph{infinite}. When $G=\Ubb_m$, for $r\in\Ubb_m$ we abbreviate $O_m(r):=O_{\Ubb_m}(r)$, which is the smallest positive integer $n$ such that $r^n\equiv 1\pmod{m}$ (which always exists because $\Ubb_m$ is finite). We can of course extend this notion for any $r\in\Z$ that is relatively prime with $m$, so $O_m(r)=O_{m}(r_0)$ where $r_0$ is the residue obtained after dividing $r$ by $m$.

\item The number of elements of a set $A$ is denoted by $\#A$. When $G$ is a group, $\#G$ is also called the \emph{order of $G$}. When $G$ is a finite group and $r\in G$, $O_G(r)$ divides $\#G$. Therefore, since $\#\Ubb_m=\varphi(m)$ where $\varphi$ denotes \emph{Euler's phi function}, $O_m(r)\mid\varphi(m)$ for any integer $r$ relatively prime with $m$. In particular, if $q$ does not divide $r$ then $O_q(r)\mid \varphi(q)=q-1$.

\item\label{def:np} Let $r\in\Z$ be relatively prime with $m$. Since $O_m(r)\mid\varphi(m)$, there is a unique (positive) integer $n_m(r)$ satisfying $O_m(r) n_m(r)=\varphi(m)$. Therefore, due to the definition of $O_m(r)$, $n_m(r)$ is the \emph{largest} $n\mid \varphi(m)$ such that $r^{\frac{\varphi(m)}{n}}\equiv 1\pmod{m}$.
\end{enumerate}
\end{notation}

The notion of $n_m(r)$ is not standard, but it will be very useful in the context of power residues modulo prime, as well as in characterizations of $O_m(r)$.

Euler's criterion for quadratic residues modulo prime can be easily generalized to power residues as follows (see e.g.~\cite[Thm.~3.11]{Nathanson},~\cite[Thm.~1.29]{takagi} and~\cite[Prop.~4.2.1]{Ireland}).

\begin{theorem}[Generalized Euler's criterion]\label{geneulercrit}
   Let $r\in \Z$, $p$ a prime not dividing $r$ and let $n$ be a positive integer. Then the equation $x^n\equiv r\pmod{p}$ has a solution iff
   \[
       r^{\frac{p-1}{\GCD(p-1,n)}}\equiv 1\pmod{p}.
   \]
   Even more, if the equation $x^n\equiv r\pmod{p}$ has a solution then it has $\GCD(p-1,n)$-many incongruent solutions modulo $p$ in total.
\end{theorem}

As a consequence,

\begin{corollary}\label{gennpred}
   Let $r\in \Z$ and $p$ a prime not dividing $r$. Then $n_p(r)$ is the largest $n\mid p-1$ such that $r$ has an $n$-th root modulo $p$.
   Moreover, the following statements are equivalent for any positive integer $n$:
   \begin{enumerate}[(i)]
       \item $x^n\equiv r\pmod{p}$ has a solution.
       \item $ r^{\frac{p-1}{\GCD(p-1,n)}}\equiv 1\pmod{p}.$
       \item $\GCD(p-1,n)\mid n_p(r)$.
   \end{enumerate}
\end{corollary}

\begin{proof}
   The equivalence (i)$\sii$(ii) is Theorem~\ref{geneulercrit};
the equivalence (ii)$\sii$(iii) can be seen from the definition of $n_p (r)$ (see Notation~\ref{notat}(\ref{def:np})).
\end{proof}

In this view, $n_p(r)$ plays a very important role in relation with power residues modulo $p$.\smallskip

The main results of this paper are divided in two parts, the first about quadratic reciprocity, and the second about power reciprocity modulo prime.

\subsection*{Main results 1: On quadratic residues}

Fix $r\in\Z$. When $p$ is an odd prime not dividing $r$ (i.e.\ $\GCD (p,r)=1$), whether $r$ is a quadratic residue modulo $p$ is determined by the \emph{Legendre symbol}, which is defined by

\begin{equation}\label{legedef}
\lege{r}{p}=\left\{\begin{array}{ll}
    1 & \text{if the equation $x^2\equiv r\pmod{p}$ has a solution,} \\
    -1 & \text{otherwise.}
\end{array}\right.
\end{equation}

In the case $r=2$, the problem of whether $2$ is a quadratic residue modulo an odd prime is already solved.

\begin{theorem}[See e.g.~{\cite[Thm.~9.6]{burton}}]\label{char2}
   If $p$ is an odd prime then $\lege{2}{p}=1$ iff $p\equiv\pm 1\pmod{8}$.
\end{theorem}

We ask about similar characterizations for any integer $r$.

\begin{problem}\label{mainQ1}
    Let $r\in\Z$. Is there a positive integer $m(r)$ and a set $L(r)\subseteq\Ubb_{m(r)}$  such that, for any prime $p$ not dividing $r$, $\lege{r}{p}=1$ iff the residue of $p$ modulo $m(r)$ is in $L(r)$?

    If so, can $L(r)$ be characterized in some way?
\end{problem}

The answer to the first question should not be difficult due to the quadratic reciprocity law, but the characterization of $L(r)$ is more interesting for settling the general problem. In fact, due to the property
\begin{equation}\label{multlege}
    \lege{ab}{p}=\lege{a}{p}\lege{b}{p},
\end{equation}
the interesting case of Problem~\ref{mainQ1} is when $r$ is a prime. In this case, we proved the following main result:

\begin{teorema}[Theorem~\ref{modq}]\label{mainquadr}
   Let $q$ be a prime. Then
   \begin{enumerate}[(a)]
       \item There is only one subgroup of $\Ubb_{4q}$ with order $\frac{\#\Ubb_{4q}}{2}$ containing $-1$. This subgroup is denoted by $\Lbb_{4q}$.
       \item For any prime $p\neq q$, $\lege{q}{p}=1$ iff the residue of $p$ modulo $4q$ is in $\Lbb_{4q}$.
   \end{enumerate}
\end{teorema}

This theorem becomes a tool to calculate $\lege{r}{p}$ for any $r\in\Z$ relatively prime with $p$. This is presented in Theorem~\ref{modr} (and at the end of Section~\ref{sec:groups}).

In the case of composite $r$, due to Equation~(\ref{multlege}) an extension of Theorem~\ref{mainquadr} is reasonable when $r$ is square free. In this case we can find a subgroup $\Lbb_{4r}$ of $\Ubb_{4r}$ containing $-1$ as in (b), but in general this group is not unique as in (a). Details are presented in Theorem~\ref{modrgroup} and in the discussion that follows it.

\subsection*{Main results 2: On  power residues}

We aim to generalize the following result to  power residues.

\begin{theorem}[See e.g.~{\cite[Thm.~256]{HW}} and~{\cite{MOF}}]\label{pnorm2}
   Let $p$ be an odd prime. Then the following statements are equivalent.
   \begin{enumerate}[(i)]
       \item The equation $x^2\equiv 2\pmod{p}$ has a solution.
       \item The equation $x^2-2y^2=p$ has an integer solution.
   \end{enumerate}
\end{theorem}

This is related to the characterization of irreducible elements of the ring $\Z[\sqrt{2}]$: an odd prime $p$ in $\Z$ is still a prime in $\Z[\sqrt{2}]$ iff the equation $x^2-2y^2=p$ does not have integer solutions (see~\cite[Thm.~256]{HW}). Recall that $x^2-2y^2$ is the norm of $x+y\sqrt{2}$ in the field extension $\Q(\sqrt{2})$ of $\Q$.

For any $n\geq 2$, denote by $D^2_n(x_0,\ldots,x_{n-1})$ the norm of $x_0+x_1\sqrt[n]{2}+\ldots x_{n-1}\sqrt[n]{2^{n-1}}$ in the field extension $\Q(\sqrt[n]{2})$ of $\Q$. This norm is defined (even in a more general context) in Section~\ref{sec:prep}, but we just state here that $D^2_n(x_0,\ldots,x_{n-1})$ is an integer when $x_0,\ldots,x_{n-1}\in\Z$. So we ask whether Theorem~\ref{pnorm2} can be generalized in the following sense.

\begin{problem}\label{problemDet=p}
    Let $n> 2$ and $p$ a prime such that $p\equiv 1\pmod{n}$. Are the following statements equivalent?
    \begin{enumerate}[(1)]
       \item The equation $x^n\equiv 2 \pmod{p}$ has a solution.
       \item The equation $D^2_n(x_0,\ldots,x_{n-1})=p$ has an integer solution.
   \end{enumerate}
\end{problem}

The solution of this problem seems to rely on tools in algebraic number theory that would go beyond elementary number theory. In this terms, we managed to solve weaker versions of the problem, where in some of them (2) is replaced by $D^2_n(x_0,\ldots,x_{n-1})\equiv 0\pmod{p}$. The trivial solution of this equation is $x_0=\ldots=x_{n-1}=0$, so we aim for non-trivial solutions. On the other hand, our results deal with any integer $r$ in place of $2$, so we use a general version $D^r_n(x_0,\ldots,x_{n-1})$ of the norm (which is defined in detail in Section~\ref{sec:prep}).

\begin{teorema}[Theorem~\ref{weaker}]\label{highresmain1}
Let $p$ be a prime, $r\in\Z$, $n\in\Z^+$ and $r_0\in\Fbb_p$ such that $r\equiv r_0\pmod{p}$.
\begin{enumerate}[(a)]
    \item The polynomial $x^n-r_0$ is irreducible in $\Fbb_p[x]$ iff the equation $D^r_n(x_0,\ldots,x_{n-1})\equiv 0\pmod{p}$ does not have a non-trivial solution in the integers.

    \item If $n\geq2$ and the equation $x^n\equiv r\pmod{p}$ has a solution, then  $D^r_n(x_0,\ldots,x_{n-1})\equiv 0\pmod{p}$ has a non-trivial solution in $\Z^n$ satisfying $-p^{\frac{1}{n}}<x_i<p^{\frac{1}{n}}$ for all $0\leq i<n$.
\end{enumerate}
\end{teorema}

The proof of Theorem~\ref{highresmain1}(b) is inspired in the proof of Theorem~\ref{pnorm2} presented in the post~\cite{MOF}.
As a consequence, we obtain the following equivalence when $n$ is a prime.

\begin{corolario}[Corollary~\ref{weakerq}]
   Let $p$ and $q$ be primes, $r\in \Z$. Then the following statements are equivalent:
   \begin{enumerate}[(i)]
       \item $x^q\equiv r\pmod{p}$ has a solution.
       \item $D^r_q(x_0,\ldots,x_{n-1})\equiv 0\pmod{p}$ has a non-trivial solution.
   \end{enumerate}
\end{corolario}

We can also conclude some weakening of the implication (2)$\imp$(1) of Problem~\ref{problemDet=p}, which yields the real implication when $n$ is a prime.

\begin{teorema}[Theorem~\ref{onedir}]\label{highresmain2}
Assume that $p$ is a prime, $n\geq 2$, $r\in\Fbb_p$ and $r_0\in\Fbb_p$ such that $r\equiv r_0\pmod{p}$. If the polynomial $x^n-r_0$ is irreducible in $\Fbb_p[x]$ then $D^r_n(\bar{x})=p$ does not have a solution in the integers.

In particular, (2)$\imp$(1) of Problem~\ref{problemDet=p} is true when $n$ is a prime.
\end{teorema}

We also present a simple proof of Theorem~\ref{pnorm2} using Theorem~\ref{highresmain1} (see~Theorem~\ref{D2}), where $2$ can also be replaced by $r\in\{-2,-1\}$. This shortens the proof in~\cite{MOF} a little bit.

We remark that ``$x^n-r$ is irreducible in $\Fbb_p[x]$" is stronger than ``$x^n\equiv r\pmod{p}$ does not have a solution". For instance, if $p\in\{7,17,23,31,41,47,71\}$, the equation $x^2\equiv 2\pmod{p}$ has a solution, but $x^{p-1}\equiv 2\pmod{p}$ does not have one. On the other hand, if $a_0$ is a solution of $x^2-2=0$ in $\Fbb_p$ then, in $\Fbb_p[x]$:
\[x^{p-1}-2=x^{2(\frac{p-1}{2})}-a_0^2=(x^{\frac{p-1}{2}}-a_0)(x^{\frac{p-1}{2}}+a_0).\]
This means that $x^{p-1}-2$ is reducible in $\Fbb_p[x]$. More details about the irreducibility of $x^n-r$ are presented in Section~\ref{sec:prep}.

We do not have any counter-example for Problem~\ref{problemDet=p} even when $x^n-2$ is reducible in $\Fbb_p[x]$.

\subsection*{Indirect motivation}

The motivation of this work is related with the study of Mersenne primes, although we do not present explicit results about them. A \emph{Mersenne number} is an integer of the form $2^n-1$ with $n\in\Z^+$ (positive integer), and a \emph{Mersenne prime} is a primer number of this form. It is well known that, whenever $2^n-1$ is a prime, $n$ must be a prime. Another curious fact is that, whenever $2^n-1$ is a Mersenne prime, there is only one (odd) prime $p$ such that $O_p(2)\mid n$, that is, such that $2^n\equiv 1\pmod{p}$. Even more, since $n$ must be prime, $n=O_p(2)$. The converse situation is interesting: if $n$ is a prime and there is only one prime $p$ such that $O_p(2)\mid n$, then $2^n-1=p^e$ for some $e\in\Z^+$. Hence, when $e=1$, $2^n-1$ is a Mersenne prime; but if $e>1$ then $p$ is a \emph{Wieferich prime}, i.e., a prime number $p$ satisfying $2^{p-1}\equiv 1\pmod{p^2}$. Recall that so far only two Wieferich primes are known, namely $1093$ and $3511$, and Silverman proved under the abc-conjecture that there are infinitely many non-Wieferich primes~\cite{silverman}.

The previous observation indicates that understanding $O_p(2)$ would lead to a better understanding of Mersenne primes and would trigger possible characterizations. On the other hand, since $O_p(2)$ is associated with $n_p(2)$, according to Corollary~\ref{gennpred} we can discover a lot about $n_p(r)$ in general by studying power residues modulo $p$.


Concerning $O_p(r)$ for some fixed integer $r>1$, the pattern of the sequence of $O_p(r)$ for prime $p$ relatively prime with $r$ seems to be very \emph{erratic}~\cite{CPmult}, but $O_n(r)$ in general can be determined in terms of $O_p(r)$ for prime $p\mid n$, see Theorems~\ref{Op^e}--\ref{Oalpha}. In particular, $O_{p^e}(r)$ is deeply related  with Wieferich primes (in base $r$). A more detailed discussion is presented in Section~\ref{sec:ord}.


\subsection*{Structure of the paper}
\ \smallskip

\noindent\textbf{Section~\ref{sec:ord}.} We discuss some simple aspects related with $O_m(r)$ and $n_p(r)$. In particular, we show expressions of $O_m(r)$ for composite $m$, and a method to obtain $n$-th roots of $1$ modulo a prime $p$, in particular $n_p(r)$-th roots of $1$. The contents of this section are known and unrelated with the main results, but we present them in accordance with the ``indirect motivation" above.\smallskip

\noindent\textbf{Section~\ref{sec:groups}.} This is dedicated to the proof of Theorem~\ref{mainquadr} and to further discussions about groups associated with quadratic reciprocity.\smallskip

\noindent\textbf{Section~\ref{sec:prep}.} We present some preliminaries in algebra that are going to be required in the proof of the main results about power residues modulo prime.\smallskip

\noindent\textbf{Section~\ref{sec:partial}.} We prove our main results about power residues modulo prime, in particular Theorems~\ref{highresmain1} and~\ref{highresmain2}.\smallskip

\noindent\textbf{Section~\ref{sec:disc}.} We discuss research related to this work.

\subsection*{Acknowledgments.} We would like to thank the anonymous referee for careful reading of the paper and for pointing out mistakes and unclear parts, which helped to improve the presentation a lot.

\section{Multiplicative order}\label{sec:ord}


We first show how the multiplicative order modulo composite numbers can be calculated.


\begin{theorem}[See e.g.~{\cite[\S 3.2, Thm.~3.6]{Nathanson}}]\label{Op^e}
   Let $p$ be an odd prime and $r\in\Z$, $r\neq\pm1$ relatively prime with $p$. Assume that $e_0$ is the maximum integer such that $O_{p^{e_0}}(r)=O_p(r)$. Then, for any $e\geq 1$,
   \[O_{p^e} (r)=\left\{\begin{array}{ll}
       O_p (r) & \text{when $e\leq e_0$},\\
       p^{e-e_0} O_p (r) &  \text{otherwise.}
       \end{array}\right.\]
\end{theorem}

The previous result has a deep connection with Wieferich primes. In fact, an odd prime $p$ is a \emph{Wieferich prime in base $r$} if $p\nmid r$ and $O_{p^2}(r)=O_p(r)$.\footnote{The standard definition is $r^{p-1}\equiv 1\pmod{p^2}$, which is equivalent thanks to Theorem~\ref{Op^e}: If $O_{p^2}(r)\neq O_p(r)$ then $O_{p^2}(r)=pO_p(r)$, which does not divide $p-1$.} Very few of these numbers are known for each $r>1$.

The following is a version of Theorem~\ref{Op^e} for $p=2$. The proof is almost the same, so we omit it.

\begin{theorem}\label{O2^e}
Assume $r\in\Z$ is odd, $r\neq\pm1$. If $e_0\geq2$ is the maximum integer such that $O_{2^{e_0}}(r)=O_4(r)$ then, for any $e\geq2$,
\[O_{2^e} (r)=\left\{\begin{array}{ll}
       O_4 (r) & \text{when $e\leq e_0$},\\
       2^{e-e_0} O_4 (r) &  \text{otherwise.}
       \end{array}\right.\]
\end{theorem}

Now we look at the case when $m>1$ is composite but not a prime power, so we assume that it has prime factorization $m=\prod_{i=1}^s p_i^{e_i}$ ($s\geq 2$).

\begin{theorem}\label{Oalpha}
When $\GCD(r ,m)$=1,
$O_m (r)=\LCM(O_{p_1^{e_1} }(r),O_{p_2^{e_2}}(r), \ldots,O_{p_s^{e_s}}(r))$.
\end{theorem}

\begin{proof}
   Let us suppose $b: =\LCM(O_{p_1^{e_1}}(r), O_{p_2^{e_2}}(r),\ldots,O_{p_s^{e_s}}(r))$. We need to prove the following.
   \begin{enumerate}[(1)]
       \item $r^b\equiv 1 \pmod{m}$.\\
       For any $i\leq s$ we know that $r^{O_{p_i^{e_i}}(r)}\equiv 1 \pmod{p_i^{e_i}}$ and $O_{p_i^{e_i}}(r)\mid b$, so $r^b\equiv 1 \pmod{p_i^{e_i}}$, i.e. $p_i^{e_i}\mid r^b -1$. Since $p_i^{e_i}$ and $p_j^{e_j}$ are relatively prime when $i\neq j$, we conclude that $m\mid r^b -1$.

       \item \emph{$b$ is the minimal number satisfying the equation $r^x \equiv 1\pmod{p}$}\\
       Assume $r^x \equiv 1\pmod{m}$. This implies $r^x \equiv 1\pmod{p_i^{e_i}}$ for any $i\leq s$, so $O_{p_i^{e_i}}(r)\mid x$. Therefore $b\mid x$, so by (1) $b$ is the minimum we claim.\qedhere
   \end{enumerate}
\end{proof}
Notice that, by the Chinese remainder theorem, the map $\Z_m\to\bigoplus_{i=1}^s\Z_{p_i^{e_i}}$ that sends $a$ to the tuple $(a_1,\ldots,a_s)$ of residues modulo $p_i^{e_i}$ is a ring isomorphism, and when restricted to $\Ubb_m$ it gives a group isomorphism onto $\oplus_{i=1}^s\Ubb_{p_i^{e_i}}$. So the previous result can be seen as a particular case of the following fact: \emph{if $G=\bigoplus_{i=1}^k G_i$ is a direct sum of groups of finite order and $\bar{a}=(a_1,\ldots,a_k)\in G$, then $O_G(\bar{a})=\LCM(O_{G_1}(a_1),\ldots,O_{G_k}(a_k))$.} (A similar proof works.)

As a consequence, we obtain the following modular equation using Euler's phi function.

\begin{corollary}\label{Oalphapre}
If $\GCD(r,m)=1$ and
 \[c = \frac{\varphi(m)}
 {\GCD(\varphi(p_1^{e_1}),\varphi(p_2^{e_2}),\ldots,\varphi(p_s^{e_s}))}\]
 Then $r^c \equiv 1\pmod{m}$.
\end{corollary}

\begin{proof}
   Since $\LCM(a_1, a_2, \dots , a_m)\cdot \GCD(a_1, a_2, \dots , a_m)\mid a_1a_2\cdots a_m$, by Theorem~\ref{Oalpha} we can prove that
   \[\begin{split}
          & O_m (r)\mid \LCM(\varphi ({p_1}^{e_1}), \varphi ({p_2}^{e_2}), \dots , \varphi ({p_s}^{e_s}))\\
       \text{and } & \LCM(\varphi ({p_1}^{e_1}), \varphi ({p_2}^{e_2}), \dots , \varphi ({p_s}^{e_s}))\mid c.
   \end{split}\]
   The theorem follows immediately.
\end{proof}

The previous result can be generalized as well in the context of direct sums of groups: \emph{if $\bar{a}\in G$ and $c=\frac{\#G}{\GCD(\#G_1,\ldots,\#G_k)}$ then $\bar{a}^c=1_G$, i.e.\ $O_G(\bar{a})\mid c$.}



From here until the end of this section, we assume that $p$ is a prime and $\GCD(r,p)=1$. We look at the effect of the power of $O_p(r)$ in $\Fbb_p^\times$, namely, properties of $k^{O_p(r)}$ for $k\in\Fbb_p$. In fact, these properties come from more general general results. First, we show that $\{k^{O_p(r)}:\, k\in\Fbb_p^\times\}$ gives the full set of $n_p(r)$-th roots of $1$ modulo $p$, which can be generalized as follows.

\begin{theorem}\label{primgen}
   Let $n\geq 1$ be an integer. Then all the $n$-th roots of unity modulo $p$ can be obtained from the set
   \[A:=\big\{a^{\frac{p-1}{\GCD(n,p-1)}}:\, a\in\Fbb_p^\times\big\}\]
   Moreover, if $r_p$ is a primitive root of $p$ then the set above coincides modulo $p$ with
   \[B:=\Big\{{r_p}^{\ell\frac{p-1}{\GCD(n,p-1)}}:\, 0\leq\ell<\GCD(n,p-1)\Big\},\]
   and their members are pairwise incongruent modulo $p$.
\end{theorem}


\begin{proof}
   We define $m(n):=\frac{p-1}{\GCD (n,p-1)}$ and $b:={r_p}^{m(n)}$.
   For any $a\in\Fbb_p^\times$, if $a\equiv {r_p}^k\pmod{p}$ then $a^{m(n)}\equiv {r_p}^{km(n)}\pmod{p}$.
   If we put $k=d \cdot\GCD (n,p-1) +\ell$ for some $d\in\mathbb{Z}$ and $0\leq \ell<\GCD (n,p-1)$, then $km(n)=d(p-1)+\ell m(n)$.
   So we get $a^{m(n)}\equiv \big({r_p}^{m(n)}\big)^\ell\equiv b^\ell\pmod{p}$. This shows $A\subseteq B$ (modulo $p$). The converse contention is trivial.

   By Theorem~\ref{geneulercrit}, the equation $x^n\equiv 1\pmod{p}$ has exactly $\GCD (n,p-1)$-many solutions in $\Fbb_p$.
   On the other hand, since $O_p(b)=\GCD (n,p-1)$, it is clear that $\big(b^\ell\big)^n \equiv 1\pmod{p}$ for all $0\leq \ell<\GCD (n,p-1)$, and that the $b^\ell$ are pairwise incongruent modulo $p$. This shows that $B$ is the complete set of $n$-th roots of unity.
\end{proof}

\begin{corollary}\label{prim}
   The set of solutions for the equation $x^{n_p (r)}\equiv 1\pmod{p}$ (i.e.\ the set of $n_p(r)$-th roots of unity modulo $p$) is
   \[\Big\{a^{O_p(r)}:\, a\in\Fbb_p^\times\Big\}=\Big\{{r_p}^{\ell O_p (r)}:\, 0\leq \ell<n_p (r)\Big\} \text{ (modulo  $p$).}\]
\end{corollary}

Recall the following properties of roots of unity modulo $p$.

\begin{lemma}\label{root1prop}
   Let $n\geq 1$ and assume that $a$ is an $n$-th root of $1$ modulo $p$. Then:
   \begin{enumerate}[(a)]
    \item If $a\equiv 1\pmod{p}$ then $\displaystyle \sum_{i=0}^{n-1} a^i \equiv n\pmod{p}$.
    \item If $a\not\equiv 1\pmod{p}$ then $\displaystyle \sum_{i=0}^{n-1} a^i \equiv 0\pmod{p}$.
\end{enumerate}
\end{lemma}
\begin{proof}
   Property (a) is trivial; since
   \[(a-1)\sum_{i=0}^{n-1}a^i
       =a^n-1\equiv 0\pmod{p},\]
   it is clear that $a\not\equiv 1\pmod{p}$ implies (b).
\end{proof}

As a consequence, we can show the behaviour of the sum of $k^{O_p(r)}$ for $1\leq k\leq p-1$, or even more generally:

\begin{theorem}[See e.g.~{\cite[Pg.~67]{takagi}}]\label{sumsgen}
Let $n\in\Z^+$ . Then:
\begin{enumerate}[(a)]
    \item $p-1\mid n \sii \displaystyle \sum_{k=1}^{p-1} k^n \equiv p-1\pmod{p}$.
    \item $p-1\nmid n  \sii \displaystyle \sum_{k=1}^{p-1} k^n \equiv 0\pmod{p}$.
\end{enumerate}
\end{theorem}

\begin{proof}
  Fix a primitive root $r_p$ of $p$, and for each $1\leq k<p$ choose $e_k<p-1$ such that ${r_p}^{e_k}\equiv k\pmod{p}$.
   We have the following:
   \[\sum_{k=1}^{p-1} k^n \equiv \sum_{k=1}^{p-1} ({r_p}^{e_k})^n\equiv \sum_{k=1}^{p-1}\big({r_p}^n\big)^{e_k}\equiv\sum_{i=0}^{p-2}\big({r_p}^n\big)^i\pmod{p}.\]
   Note that any member of $\Fbb_p^\times$ is a $(p-1)$-th root of $1$, so we can apply Lemma~\ref{root1prop} to conclude:
   \begin{enumerate}[(a)]
       \item if ${r_p}^n\equiv 1\pmod{p}$ then $\displaystyle\sum_{i=0}^{p-2}\big({r_p}^n\big)^i\equiv p-1\pmod{p}$;

       \item if ${r_p}^n\not\equiv 1\pmod{p}$ then  $\displaystyle\sum_{i=0}^{p-2}\big({r_p}^n\big)^i
       \equiv 0\pmod{p}$.
   \end{enumerate}
   It is easy to verify that ${r_p}^n\equiv 1\pmod{p}$ is equivalent to $p-1\mid n$, so the result follows.
\end{proof}

\begin{corollary}\label{sums}
Let $r\in\Z$ such that $\GCD(r,p)=1$. Then:
\begin{enumerate}[(a)]
    \item $O_p (r)=p-1 \sii \displaystyle \sum_{k=1}^{p-1} k^{O_p (r)} \equiv p-1\pmod{p}$.
    \item $O_p (r)\neq p-1  \sii \displaystyle \sum_{k=1}^{p-1} k^{O_p (r)} \equiv 0\pmod{p}$.
\end{enumerate}
\end{corollary}

\section{Groups associated with quadratic residues}\label{sec:groups}

This section is dedicated to the proof of Theorem~\ref{mainquadr}.


Recall the Legendre symbol $\lege{r}{p}$ as presented in Equation~(\ref{legedef}). It is known that the map $\Fbb^{\times}_p\to\Ubb_4$, $r\mapsto \lege{r}{p}$ is a group homomorphism, where $\Ubb_4 =\{1,-1\}$ as a multiplicative group,\footnote{This is isomorphic to the additive group $\Z_2$.} so
\begin{equation}\label{Lq}
    \Lbb^*_p:=\Bigg\{a\in\Fbb^\times_p:\, \lege{a}{p}=1\Bigg\}
\end{equation}
is a subgroup of $\Fbb^{\times}_p$ of order $\frac{p-1}{2}$ (half of the order of $\Fbb^{\times}_p$).

We look at the following converse situation: given an integer $r$, characterize the odd primes $p$ relatively prime with $r$ such that $\lege{r}{p}=1$. This is associated with $n_p(r)$ in the following sense.

\begin{lemma}\label{lem:np}
  Let $p$ be an odd prime, $r\in\Z$ such that $\GCD(r,p)=1$. Then the following statements are equivalent:
  \begin{enumerate}[(i)]
      \item $\lege{r}{p}=1$.
      \item $x^2\equiv r\pmod{p}$ has a solution.
      \item $r^{\frac{p-1}{2}}\equiv 1\pmod{p}$.
      \item $n_p(r)$ is even.
  \end{enumerate}
\end{lemma}

\begin{proof}
The equivalence (i)$\sii$(ii) follows from the definition of Lagrange's symbol.
The others are a direct consequence of Corollary~\ref{gennpred} (applied to $n=2$).
\end{proof}

First, we look at the case when $r=q$ is a prime. If $q=2$ we have the following situation.

\begin{theorem}\label{mod8}
If $p$ is an odd prime then the following statements are equivalent.
\begin{enumerate}[(i)]
    \item $\lege{2}{p}=1$.
    \item $p\equiv \pm 1\pmod{8}$.
    \item $2^{\frac{p-1}{2}}\equiv 1\pmod{p}$.
    \item $n_p(2)$ is even.
\end{enumerate}
\end{theorem}

\begin{proof}
   (i)$\sii$(ii) is known, see Theorem~\ref{char2}. The rest follows by Lemma~\ref{lem:np}.
\end{proof}

We aim to generalize Theorem~\ref{mod8} for any $r$ in the place of $2$, concretely, to find a condition like in (ii) that characterizes $\lege{r}{p}$ for any odd prime $p$ relatively prime with $r$.

An observation about the case $r=2$: Denote $\Lbb_8:=\{1,-1\}$ as a subgroup of $\Ubb_8$. Note that this is the only subgroup of $\Ubb_8$ of order $2$ (half of the order of $\Ubb_8$) that contains $-1$. Theorem~\ref{mod8} says that $\lege{2}{p}=1$ iff $p\equiv c\pmod{8}$ for some $c\in\Lbb_8$, which validates Theorem~\ref{mainquadr} for $r=2$.\medskip

Assume that $r=q$ is an odd prime. If $p\neq q$ is an odd prime then, by the quadratic reciprocity law:
\begin{equation}\label{lgq1}
    \lege{q}{p}=(-1)^{\frac{q-1}{2}\frac{p-1}{2}}\lege{p}{q}.
\end{equation}

We start assuming $q\equiv -1\pmod{4}$,\footnote{Although the easy case is $q\equiv 1\pmod{4}$, we decided to start with the other case for convenience of the presentation.} in which case
\[
    \lege{q}{p}=(-1)^{\frac{p-1}{2}}\lege{p}{q}.
\]
Therefore, $\lege{q}{p}=1$ iff one of the following cases hold:
\begin{enumerate}[(i)]
    \item $p\equiv 1\pmod{4}$ and $p\equiv a\pmod{q}$ for some $a\in\Lbb^*_q$ (see Equation~(\ref{Lq})), or
    \item $p\equiv -1\pmod{4}$ and $p\equiv b\pmod{q}$ for some $b\in\Ubb_q\menos\Lbb^*_q$.
\end{enumerate}

For any odd prime $q_0$: by the Chinese remainder theorem, the map $F_{q_0}:\Z_{4 q_0}\to\Z_{4}\oplus\Fbb_{q_0}$ that sends any $x$ to the pair $(x_0,x_1)$ of remainders modulo $4$ and $q_0$ respectively, is a ring isomorphism. When this map is restricted to $\Ubb_{4 q_0}$ it becomes a group isomorphism onto $\Ubb_{4}\oplus\Fbb^\times_{q_0}$.

Coming back to our argument, using the previous terminology we conclude that $\lege{q}{p}=1$ iff $p\equiv c\pmod{4q}$ for some $c\in\Ubb_{4 q}$ such that $c$ satisfies one of the following conditions:
\begin{enumerate}[$(\star)^q_1$:]
    \item $F_q(c)=(1,a)$ for some $a\in\Lbb^*_q$ (by (i)), or
    \item $F_q(c)=(-1,b)$ for some $b\in\Ubb_q\menos \Lbb^*_q$ (by (ii)).
\end{enumerate}
Let $\Lbb_{4q}$ be the set of $c\in\Ubb_{4q}$ satisfying either $(\star)^q_1$ or $(\star)^q_2$. Since
\[L'_{(4,q)}:=\{(e,a)\in\Ubb_4\oplus\Ubb_q:\, \text{either $e=1$ and $a\in\Lbb^*_q$, or $e\neq 1$ and $a\notin\Lbb^*_q$} \}\]
is a subgroup of $\Ubb_4\oplus\Ubb_q$ and $\Lbb_{4q}$ is the inverse image under $F_q$ of this subgroup, we conclude that $\Lbb_{4q}$ is a subgroup of $\Ubb_{4q}$.

Moreover, $\Lbb_{4q}$ has order $q-1$, which is half of the order of $\Ubb_{4q}$, and $-1\in\Lbb_{4q}$: Since $\Lbb^*_q$ has order $\frac{q-1}{2}$, it is clear that the order of $L'_{(4,q)}$ is double, that is, $q-1$, and this is the order of $\Lbb_{4q}$; note that $F_q(-1)=(-1,-1)$ and $-1\notin\Lbb^*_q$ because $q\equiv -1\pmod{4}$, so it satisfies $(\star)^q_2$ and we get $-1\in\Lbb_{4q}$.\medskip


We turn to the case when $q\equiv 1\pmod{4}$. By Equation~(\ref{lgq1}) we obtain that $\lege{q}{p}=\lege{p}{q}$, so $\lege{q}{p}=1$ iff $p\equiv a\pmod{q}$ for some $a\in\Lbb^*_{q}$. Using the ring isomorphism $F_q$ introduced before, define
\[\Lbb_{4q}:=\{c\in\Ubb_{4q}:\, F_q(c)=(e,a)\text{ for some }e\in\Ubb_4\text{\ and }a\in\Lbb^*_q\}.\]
Since this is the inverse image under $F_q$ of $\Ubb_4\oplus\Lbb^*_q$ and this is a subgroup of $\Ubb_4\oplus\Ubb_q$ of size $q-1$, we conclude that $\Lbb_{4q}$ is a subgroup of $\Ubb_{4q}$ of order $q-1$ (half of the order of $\Ubb_{4q}$). Even more, $-1\in\Lbb_{4q}$ because $F_q(-1)=(-1,-1)$ and, since $q\equiv 1\pmod{4}$, $-1\in\Lbb^*_q$.

The previous argument is then summarized in the following result, which generalizes Theorem~\ref{mod8} and concludes the proof of Theorem~\ref{mainquadr}.

\begin{theorem}\label{modq}
   Let $q\neq p$ be prime numbers with $p$ odd. Then $\lege{q}{p}=1$ iff $p\equiv c\pmod{4q}$ for some $c\in\Lbb_{4q}$.

   Moreover, $\Lbb_{4q}$ is the unique subgroup of $\Ubb_{4q}$ with order $q-1$ (half of the order of $\Ubb_{4q}$) that contains $-1$.
\end{theorem}
\begin{proof}
   According to the previous discussion, it remains to show that, whenever $q$ is an odd prime, $\Lbb_{4q}$ is the unique subgroup of $\Ubb_{4q}$ as in the statement. So let $G$ be a subgroup of $\Ubb_{4q}$ of order $q-1$ with $-1\in G$. This indicates that $(-1):=\{1,-1\}$ is a subgroup of $G$, so when taking quotients
   \[\Ubb_{4q}/G\cong(\Ubb_{4q}/(-1))/(G/(-1)).\]
   Note that $\Ubb_{4q}/(-1)\cong\Ubb_{2q}$ and $G/(-1)$ is a subgroup of $\Ubb_{4q}/(-1)$ of order $\frac{q-1}{2}$. So it is enough to show that $\Ubb_{2q}$ contains only one subgroup of order $\frac{q-1}{2}$.

   By the Chinese remainder theorem, $\Ubb_{2q}$ is isomorphic to $\Ubb_2\oplus\Fbb^\times_q$, which is isomorphic to $\Fbb^\times_q$ itself. Since $\Fbb^\times_q$ is a cyclic group, it only contains one subgroup of order $\frac{q-1}{2}$, which concludes the proof.
\end{proof}

Now we turn to the more general case $r\in\Z^+$. If $r$ is a square then trivially $\lege{r}{p}=1$ for any odd prime $p$ relatively prime with $r$; if $r=\prod_{i=1}^s q_i^{e_i}$ is the prime factorization of $r$ and $r$ is not a square, and $p$ is an odd prime relatively prime with $r$, then by~(\ref{multlege}):
\[\lege{r}{p}=\prod_{i=1}^s{\lege{q_i}{p}}^{e_i}=\prod_{i\in S}\lege{q_i}{p}=\lege{\prod_{i\in S}q_i}{p}\]
where $S:=\{i:\, e_i\text{ is odd}\}$.

Therefore, the general case reduces to when $r$ is square free, that is, it has its prime factorization of the form $q_1\cdots q_m$ (when all prime powers are $1$). Since
\[\lege{r}{p}=\prod_{i=1}^m\lege{q_i}{p}\]
we obtain that $\lege{r}{p}=1$ iff the number of elements of the set $\big\{i:\, \lege{q_i}{p}=-1\big\}$ is even. We can express this in terms of the groups $\Lbb_{4q}$ thanks to Theorem~\ref{modq}.

\begin{theorem}\label{modr}
   Let $r\in\Z^+$.
   \begin{enumerate}[(a)]
       \item If $r$ is a square then $\lege{r}{p}=1$ for any odd prime $p$ with $\GCD(p,r)=1$.
       \item Assume that $r$ is not a square and $r=\prod_{i=1}^s q_i^{e_i}$ is its prime factorization. If $S:=\{i:\, e_i\text{ is odd}\}$ then, for any odd prime $p$ with $\GCD(p,r)=1$, $\lege{r}{p}=1$ iff the number of elements of the set
       \[\{i\in S:\, p\equiv b\pmod{q_i}\text{ for some }b\in\Ubb_{4q_i}\menos\Lbb_{4q_i}\}\]
       is even.
   \end{enumerate}
\end{theorem}

We develop the case $r=q_1\cdots q_m$ (prime factorization) a bit more. Consider the ring homomorphism $F'_r:\Z\to\bigoplus_{i=1}^m\Z_{4q_i}$ that sends $x$ to the tuple $(x_1,\ldots, x_m)$ where $x\equiv x_i\pmod{4q_i}$ for any $i$. Although the kernel of this map is $(4r)\Z$, the image is not everything: as a consequence of the Chinese remainder theorem (for non-coprime moduli),\footnote{This holds even when some $q_i$ is $2$. Recall that the Chinese remainder theorem (for non-coprime moduli) states that a system of congruences $x\equiv a_i\pmod{n_i}$ ($1\leq i\leq m$) has a solution iff $a_i\equiv a_j\pmod{\gcd(n_i,n_j)}$ for all $i,j$, and the solution (if it exists) is unique modulo $\LCM(n_1,\ldots,n_m)$ (this is a generalization of~\cite[\S2.4, Thm.~2.9]{Nathanson} that can be easily proved by induction).}
\[F'_r[\Z]=\Bigg\{(x_1,\ldots,x_m)\in\bigoplus_{i=1}^m\Z_{4q_i}:\, x_i\equiv x_j\pmod{4}\text{ for all }i,j\Bigg\}.\]
Therefore, the map $F_r:\Z_{4r}\to F'_r[\Z]$ defined by $F_r(a)=F'_r(a)$, is a ring isomorphism. If we restrict this map to $\Ubb_{4r}$, we get a group isomorphism onto
\[U'_{(4,r)}:=F'_r[\Z]\cap\bigoplus_{i=1}^m\Ubb_{4q_i}=\Bigg\{(x_1,\ldots,x_m)\in\bigoplus_{i=1}^m\Ubb_{4q_i}:\, x_i\equiv x_j\pmod{4}\text{ for all }i,j\Bigg\}.\]
According to (b), define
\[    L'_{(4,r)}:=\{(x_1,\ldots,x_m)\in U'_{(4,r)}:\, \text{the number of elements of the set }
    \{i:\, x_i\in\Ubb_{4q_i}\menos\Lbb_{4q_i}\} \text{ is even}\}.\]
And let $\Lbb_{4r}=\{x\in\Ubb_{4r}:\, F_r(x)\in L'_{(4,r)}\}$. Therefore, for any odd prime $p$ with $\GCD(p,r)=1$, $\lege{r}{p}=1$ iff $p\equiv c\pmod{4r}$ for some $c\in\Lbb_{4r}$.

It is easy to check that $L'_{(4,r)}$ is a subgroup of $U'_{(4,r)}$ of half order, so $\Lbb_{4r}$ is a subgroup of $\Ubb_{4r}$ of half order. Moreover, $-1\in\Lbb_{4r}$ because $\{i:\, -1\in\Ubb_{4q_i}\menos\Lbb_{4q_i}\}$ is empty by Theorem~\ref{modq} (so it has zero elements). To summarize:

\begin{theorem}\label{modrgroup}
  Let $r\in\Z^+$ with prime factorization $r=q_1\cdots q_m$. Then there is a subgroup $\Lbb_{4r}$ of $\Ubb_{4r}$ of half order, containing $-1$, such that for any odd prime $p$ with $\GCD(p,r)=1$, $\lege{r}{p}=1$ iff $p\equiv c\pmod{4r}$ for some $c\in\Lbb_{4r}$.
\end{theorem}

However, it may be that $\Lbb_{4r}$ is not the only subgroup of $\Ubb_{4r}$ of half order containing $-1$. For example, consider $r=15$: $\Lbb_{60}=\{\pm 1,\pm 7, \pm 11,\pm 17\}$, but $\{\pm 1,\pm 11, \pm 19, \pm29\}$ is another subgroup of $\Ubb_{60}$ of half order containing $-1$.

To finish this section, we consider negative integers. If $r\in\Z^+$ and $p$ is an odd prime with $\GCD(r,p)=1$ then
\[\lege{-r}{p}=\lege{-1}{p}\lege{r}{p}.\]
Since $\lege{-1}{p}=1$ iff $p\equiv 1\pmod{4}$, $\lege{-r}{p}$ can be easily calculated by Theorem~\ref{modr}.

\section{Preliminaries about modules and fields}\label{sec:prep}

Throughout this section fix an arbitrary integral domain $R$, $r\in R$ and a natural number $n$. We first discuss the ring quotient $R^r_n:=R[x]/(x^n-r)$. It is very common to look at this ring quotient when $R$ is a field and $x^n-r$ is irreducible in $R[x]$, in which case $R^r_n$ is a field. But in this work we also want to look at the situation when $x^n-r$ is reducible in $R[x]$, in which case $R^r_n$ is not an integral domain. In any case:

\begin{lemma}\label{freemod}
The ring $R^r_n$ is a free $R$-module with basis $\{1,u,\ldots,u^{n-1}\}$ where $u:=x\pmod{(x^n-r)}$, even more $R^r_n$ is an $R$-algebra.
\end{lemma}
\begin{proof}
   Recall that $R[x]$ satisfies the \emph{division algorithm with monic polynomials}: for any $f(x),g(x)\in R[x]$, if $g(x)$ is of the form $x^m+a_{m-1}x^{m-1}+\ldots+a_0$ ($m=0$ is allowed, in which case $g(x)=1$) then there are \underline{unique} $q(x),t(x)\in R[x]$ such that $f(x)=q(x)g(x)+t(x)$ and $t(x)$ has degree smaller than $g(x)$.

   Now, if $0\neq f(x)\in R[x]$ has degree smaller than $n$ then, by applying the previous division algorithm to $g(x)=x^n-r$, we obtain that $f(x)=q(x)g(x)+t(x)$ for unique $q(x)$ and $t(x)$, the latter with degree smaller than $n$. Hence $q(x)=0$: if $q(x)\neq0$ has degree $m\geq0$, then $q(x)g(x)$, and thus $f(x)$, have degree $n+m$, which contradicts that $f(x)$ has degree smaller than $n$. Therefore $t(x)=f(x)\neq0$, meaning that $f(x)$ is not a multiple of $x^n-r$ (otherwise, $t(x)=0$ by the division algorithm with monic polynomials).

   Let $R'$ be the $R$-submodule of $R[x]$ generated by $\{ 1,x,\ldots,x^{n-1}\}$, which is a free $R$-module. The previous paragraph shows that the surjective $R$-module homomorphism $R'\to R^r_n$ that sends each $x^i$ to $u^i$ has kernel equal to the zero ring, so it is an $R$-module isomorphism. This shows that $R^r_n$ is a free $R$-module with basis $\{1,u,\ldots,u^{n-1}\}$.

   It is clear that $R^r_n$ is an $R$-algebra.
\end{proof}

If $x^n-r$ is reducible in $R[x]$ then $R^r_n$ is not an integral domain, but it is an integral domain when $R$ is a unique factorization domain and $x^n-r$ is irreducible in $R[x]$. In general, $R^r_n$ can be expressed as a ring of matrices $\M^r_n(R)$ such that the determinant works as the norm of the elements of the ring.

\begin{definition}\label{DefMat}
   \begin{enumerate}[(1)]
   \item For $\bar{x}=(x_0,\ldots,x_{n-1})\in R^n$ define
       \[M^r_n(\bar{x}):=\begin{pmatrix}
    x_0 & rx_{n-1} & rx_{n-2} & \dots & rx_2 & rx_1 \\
    x_1 &    x_0   & rx_{n-1} & \dots & rx_3 & rx_2 \\
    \vdots & \vdots & \vdots  &\ddots & \vdots & \vdots \\
    x_{n-2} & x_{n-3} & x_{n-4} & \dots &  x_0 & rx_{n-1}\\
    x_{n-1} & x_{n-2} & x_{n-3} & \dots & x_1 & x_0
    \end{pmatrix}\]
    and denote its determinant by $D^r_n(\bar{x})$.

    \item If $z\in R^r_n$ we denote $M^r_n(z):=M^r_n(\bar{x})$ and $D^r_n(z):=D^r_n(\bar{x})$ where $\bar{x}=(x_0,\ldots,x_{n-1})\in R^n$ is the unique tuple such that $z=\sum_{i=0}^{n-1}x_i u^i$.

    \item Define $\M^r_n(R):=\{M^r_n(\bar{x}):\, \bar{x}\in R^n\}$. When $R$ is understood from the context we just write $\M^r_n$.
   \end{enumerate}
\end{definition}

These matrices actually describe the shift endomorphisms in $R^r_n$:

\begin{lemma}\label{mathomom}
   If $z\in R^r_n$ then the matrix $M^r_n(z)$ characterizes the endomorphism $R^r_n\to R^r_n$ given by $w\mapsto zw$. Concretely, $M^r_n(z)$ is the unique matrix with the following property: if $w=\sum_{i=0}^{n-1}x_iu^i$ for some $\bar{x}\in R^n$, then $zw=\sum_{i=0}^{n-1}y_iu^i$ where $\bar{y}=M^r_n(z)\bar{x}$.
\end{lemma}

As a consequence $\M^r_n$ is a subring of the ring of $n\times n$ matrices with entries in $R$, even more, $\M^r_n$ is commutative and so it is an $R$-algebra. In fact, it characterizes $R^r_n$.

\begin{lemma}\label{Frnisom}
   The function $M^r_n:R^r_n\to\M^r_n$ is an $R$-algebra isomorphism, and the map $D^r_n:R^r_n\to R$ satisfies $D^r_n(zz')=D^r_n(z)D^r_n(z')$ for any $z,z'\in R^r_n$.
\end{lemma}

The function $D^r_n$ has the role of a \emph{norm} for $R^r_n$. In fact,
when $F$ is a field and $x^n-r$ is irreducible in $F[x]$, $F^r_n$ is a field and $D^r_n$ is its norm as an $F$-extension.

We list the exact form of some few $D^r_n(\bar{x})$ with $\bar{x}\in R^n$:
\[\begin{split}
    D^r_2(\bar{x}) = & {x_0}^2- {x_1}^2 r; \\
    D^r_3(\bar{x}) =  & {x_0}^3 + {x_1}^3 r + {x_2}^3 r^2 - 3 x_0 x_1 x_2 r;\\
    D^r_4(\bar{x}) = & {x_0}^4 - {x_1}^4 r + 4 x_0 {x_1}^2 x_2 r - 2 {x_0}^2 {x_2}^2 r - 4 {x_0}^2 x_1 x_3 r + {x_2}^4 r^2 -
 4 x_1 {x_2}^2 x_3 r^2 + \\
   & 2 {x_1}^2 {x_3}^2 r^2 + 4 x_0 x_2 {x_3}^2 r^2 - {x_3}^4 r^3.
\end{split}\]

We can also talk about conjugates in $R^r_n$. In field extensions like $\Q(i)$ and $\Q(\sqrt{2})$, the conjugate $\bar{z}$ of some element $z$ satisfies that $z\bar{z}$ is the norm of $z$. In the general case we can look at the matrix characterization: for any matrix $A$ of dimensions $n\times n$ (with entries in $R$), $A\cdot\adj(A)=|A|I_n$ where $I_n$ is the identity matrix of dimensions $n\times n$, $\adj(A)$ is the \emph{adjugate} of $A$ and $|A|$ is the \emph{determinant} of $A$. Since the determinant acts as a norm, then $\adj(A)$ works as the (analog of the) \emph{conjugate} of $A$. Recall that the matrix $A$ is \emph{invertible} if there is some unique matrix $A^{-1}$ of dimensions $n\times n$, with entries in $R$, such that $AA^{-1}=A^{-1}A=I_n$. Recall that $A$ is invertible iff $|A|$ is a unit in $R$, in which case $A^{-1}=|A|^{-1}\adj(A)$. In $\M^r_n(R)$ we obtain:

\begin{lemma}\label{matinv}
   If $A\in\M^r_n(R)$ then $\adj(A)\in\M^r_n(R)$. In particular, if $A\in\M^r_n(R)$ is invertible (as a matrix) then $A^{-1}\in\M^r_n(R)$.
\end{lemma}
\begin{proof}
   An analog of the Caley-Hamilton Theorem indicates that
   \[(-1)^{n-1}\adj(A)=A^{n-1}+c_{n-1}A^{n-2}+\cdots+c_1I_n\]
   where $c_{n-1},\ldots,c_0\in R$ and $\lambda^n+c_{n-1}\lambda^{n-1}+\cdots+c_0$ is the characteristic polynomial of $A$. If $A\in\M^r_n$ then $(-1)^{n-1}\adj(A)\in\M^r_n$ by the expression above, so $\adj(A)\in\M^r_n$.

   In particular, when $A$ is invertible, $A^{-1}=|A|^{-1}\adj(A)\in\M^r_n$.

   We also present an elementary proof in the case when $A\in\M^r_n(R)$ is invertible as a matrix with entries in $F$, where $F$ is the field of fractions of $R$. Choose $z\in R$ such that $A=M^r_n(z)$. Since $A$ is invertible, by Lemma~\ref{mathomom} the map $w\mapsto zw$ is an automorphism on $F^r_n$, so there is some $z'\in F$ such that $z z'=1$, hence $w\mapsto z'w$ is the inverse of the previous map. Therefore $A^{-1}=M^r_n(z')\in\M^r_n(F)$, which implies that $\adj(A)=|A|A^{-1}\in\M^r_n(F)$. But $\adj(A)$ is a matrix with entries in $R$, so $\adj(A)\in\M^r_n(R)$.
\end{proof}

Now that we know a bit more about the structure of $R^r_n$,
we now look at sufficient and necessary conditions for the polynomial $x^n-r$ to be irreducible.

\begin{lemma}\label{irrnecc}
   If $x^n-r$ is irreducible in $R[x]$ then:
   whenever $q\mid n$ is prime, $x^q-r=0$ does not have a solution in $R$.
\end{lemma}
\begin{proof}
   Assume that $q\mid n$ is prime and $x^q-r=0$ has a solution $v$ in $R$, that is,
   $v^q=r$ in $R$. Then, in $R[x]$,
   \[x^n-r=x^{q\frac{n}{q}}-v^q=(x^{\frac{n}{q}}-v)(x^{\frac{n}{q} (q-1)}+\ldots+v^{q-1}),\]
   so $x^n-r$ is reducible.
\end{proof}

We will prove the converse in some cases of interest by using the following result. From now on, fix a field $F$ and $r\in F$.

\begin{theorem}[See~{\cite[Ch.~VI~\S9]{Lang}}]\label{irrchar}
    The polynomial $x^n-r$ is irreducible in $F[x]$ iff the following two conditions hold.
    \begin{enumerate}[(i)]
        \item If $q\mid n$ is prime then the equation $x^q-r=0$ does not have a solution in $F$.
        \item If $4\mid n$ then the equation $4x^4+r=0$ does not have a solution in $F$.
    \end{enumerate}
\end{theorem}
\begin{proof}
   The cited reference states and proves that (i) and (ii) implies that $x^n-r$ is irreducible in $F[x]$. The converse implication is true for any ring $R$ and it is easy to prove. Assume that $r\in R$. Lemma~\ref{irrnecc} shows that $x^n-r$ irreducible in $R[x]$ implies (i). To show that (ii) is also implied we prove that, whenever $4\mid n$ and $4u^4+r=0$ for some $u\in R$, $x^n-r$ is reducible in $R[x]$. Since $n=4k$ for some $k\geq1$, we get
   \[x^n-r=(x^k)^4+4u^4=((x^2)^k-2ux^k+2u^2)((x^2)^k+2ux^k+2u^2).\qedhere\]
\end{proof}

\begin{corollary}\label{irrx^q}
   Let $q$ be a prime and let $F$ be a field. Then $x^q-r=0$ does not have a solution in $F$ iff $x^q-r$ is irreducible in $F[x]$.
\end{corollary}

Condition (ii) can be suppressed when we look at fields of prime characteristic.

\begin{theorem}\label{irrFp+}
    Let $p$ be a prime and assume that $4\nmid n$ or $4\mid p-1$ or $p=2$. If $F$ has characteristic $p$ then $x^n-r$ is irreducible in $F[x]$ iff, for any prime $q\mid n$, $x^q-r=0$ does not have a solution in $F$.
\end{theorem}
\begin{proof}
   We showed one direction in Lemma~\ref{irrnecc}. To see the converse, assume that, for any prime $q\mid n$, $x^q-r=0$ does not have a solution in $F$, which means that (i) of Theorem~\ref{irrchar} is valid. By using the same theorem, it is enough to show that (ii) holds, that is, the equation $4x^4+r=0$ does not have a solution in $F$ when $4\mid n$.

   Assume that $4\mid n$, so either $4\mid p-1$ or $p=2$ by hypothesis. In the case $4\mid p-1$ assume towards a contradiction that $4x^4+r=0$ has a solution $x_0\in F$. So $-r=4x_0^4=(2x_0^2)^2$. Let $y_0:=2x_0^2$, so $y_0^2=-r$.

   On the other hand, by properties of the Legendre symbol, 
   \[\left(\dfrac{-1}{p}\right)=(-1)^{\frac{p-1}{2}}=1\text{\ (because $4\mid p-1$)},\]
   which means that $-1\equiv z_0^2\pmod{p}$ for some $z_0\in\Fbb_p$. Hence, $r=(-r)(-1)=(y_0z_0)^2$, that is, the equation $x^2-r=0$ has a solution in $F$, but this is not true by hypothesis: since $2$ is prime and $2\mid n$, $x^2-r=0$ does not have a solution in $F$.

   In the case $p=2$ we have $4x^4+r=r$. If $4x^4+r=0$ has a solution in $F$ then $r=0$, but $4\mid n$ so the hypothesis says that the equation $x^2=0$ does not have a solution in $F$, which is absurd.
\end{proof}

\begin{corollary}\label{irrFp}
   Let $p$ be a prime and assume that $n\mid p-1$. If $F$ has characteristic $p$ then $x^n-r$ is irreducible in $F[x]$ iff, for any prime $q\mid n$, $x^q-r=0$ does not have a solution in $F$.
\end{corollary}
\begin{proof}
   Immediate by Theorem~\ref{irrFp+} because $4\mid n$ implies $4\mid p-1$ when $p$ is odd.
\end{proof}


In some cases, we can also characterize irreducibility of $x^n-r$ in $\Q[x]$.

\begin{theorem}\label{irrQ}
   Let $n$ be a natural number. If $r\in\Q$ and $r>0$ then $x^n-r$ is irreducible in $\Q[x]$ iff $x^q-r=0$ does not have a solution in $\Q$ for any prime $q\mid n$.
\end{theorem}
\begin{proof}
   This is a direct consequence of Theorem~\ref{irrchar} since condition (ii) there is always satisfied.
\end{proof}

The previous result actually applies to any ordered field.

To finish this section, we show that irreducible in $\Fbb_p[x]$ is stronger than irreducible in $\Q[x]$ when $r\in\Z$.

\begin{corollary}\label{irrQFp}
  Let $p$ be a prime, $r\in\Z$ and $n\in\Z^+$. If $r\equiv r_0\pmod{p}$ and $x^n-r_0$ is irreducible in $\Fbb_p[x]$ then $x^n-r$ is irreducible in $\Q[x]$.
\end{corollary}
\begin{proof}
   Assume that $x^n-r_0$ is irreducible in $\Fbb_p[x]$. We first prove that $x^q-r=0$ does not have a solution in $\Q$ for any prime $q\mid n$.
   Using Lemma~\ref{irrnecc} with $R=\mathbb{F}_p$, we know that $x^q -r_0=0$ does not have a solution in $\mathbb{F}_p$ for any prime $q\mid n$, which implies that the equation $x^q-r=0$ does not have a solution in $\Z$, so neither in $\Q$: if $a,b\in\Z$ are relative prime, $b>0$, and $\big(\frac{a}{b}\big)^q-r=0$, then $a^q=rb^q$, which implies that $b=1$ (if $b>1$ then $r=0$, so $a=0$ and, since $\gcd(a,b)=1$, $b=1$, contradiction), thus $x^q-r$ has a solution in $\Z$.

   In the case $r>0$ the result follows by Theorem~\ref{irrQ}; in the case $n\nmid 4$, the result follows by Theorem~\ref{irrchar}; and when $r=0$, we must have $n=1$ (because we assumed $x^n-r_0$ irreducible in $\Fbb_p[x]$) and then $x^n-r=x$ is irreducible in $\Q$.

   So it remains to consider the case when $r<0$ and $n\mid 4$. Here it remains to show that (ii) of Theorem~\ref{irrchar} holds for $F=\Q$. Towards a contradiction, assume that $4a^4+r=0$ for some $a\in\Q$. Since $r\in\Z$ and $a^4=\frac{-r}{4}$, we must have that $a\in\Z$. Therefore, modulo $p$ we get that $4x^4+r_0=0$ has a solution in $\Fbb_p$, but this contradicts (ii) of Theorem~\ref{irrchar} for $x^n-r_0$ in $\Fbb_p[x]$.
\end{proof}



\section{Power residues}\label{sec:partial}


In this section we show the main results concerning power residues. We start with Theorem~\ref{highresmain1}.

\begin{theorem}\label{weaker}
   Let $p$ be a prime, $n\in\Z^+$, $r\in\Z$ and let $r_0\in\Fbb_p$ such that $r\equiv r_0\pmod{p}$. 
   \begin{enumerate}[(a)]
       \item The polynomial $x^n-r_0$ is irreducible in $\Fbb_p[x]$ iff the equation $D^r_n(x_0,\ldots,x_{n-1})\equiv 0\pmod{p}$ does not have a \underline{non-trivial} solution in the integers.

       \item If $x^n-r$ is reducible in $\Q[x]$ then $D^r_n(\bar{x})=0$ has a non-trivial solution in the integers.

       \item If $n\geq2$ and the equation $x^n\equiv r\pmod{p}$ has a solution, then  $D^r_n(x_0,\ldots,x_{n-1})\equiv 0\pmod{p}$ has a non-trivial solution in the integers. Even more, this solution satisfies $-p^{\frac{1}{n}}<x_i<p^{\frac{1}{n}}$ for all $0\leq i<n$.
   \end{enumerate}
\end{theorem}
\begin{proof}
  Set $F:=\Fbb_p$. We first show (a). Assume that $x^n-r_0$ is irreducible in $F[x]$. Then $F^{r_0}_n=F(u)$ is a field extension of $F$ with $u:=\sqrt[n]{r_0}$, which is isomorphic to $\M^{r_0}_n(F)$ by Lemma~\ref{Frnisom}.
  Let $\bar{x}=(x_0,\ldots,x_{n-1})\neq(0,\ldots,0)$ with $x_i\in\Fbb_p$ ($0\leq i<n$), and set $A:=M^{r_0}_n(\bar{x})$. By Lemma~\ref{matinv} $A^{-1}\in\M^{r_0}_n$, so $D^{r_0}_n(\bar{x})\neq0$ in $\Fbb_p$, that is, $D^r_n(\bar{x})\not\equiv0\pmod{p}$.\medskip

  For the converse, assume that $x^n-r_0$ is reducible in $F[x]$. Then $F^{r_0}_n$ is not an integral domain, so there are non-zero $z,w\in F^{r_0}_n$ such that $zw=0$. Then, by Lemma~\ref{Frnisom}, $D^r_n(z)D^r_n(w)\equiv 0\pmod{p}$, so either $D^r_n(z)\equiv 0\pmod{p}$ or $D^r_n(w)\equiv 0\pmod{p}$.\medskip

  To see (b): if $x^n-r$ is reducible in $\Q[x]$ then there are non-zero $z,w\in\Q^r_n$ such that $zw=0$. Even more, we can find non-zero vectors $\bar{x},\bar{y}\in\Z^n$ such that $z'w'=0$ where $z'=\sum_{i=0}^{n-1}x_iu^i$ and $w'=\sum_{i=0}^{n-1}y_iu^i$ (here $u$ determines the basis of $\Q^r_n$ as a $\Q$-vector space). Therefore $D^r_n(\bar{x})D^r_n(\bar{y})=0$, so $D^r_n(\bar{x})=0$ or $D^r_n(\bar{y})=0$.\medskip

  Now we show (c). Assume that $x^n\equiv r \pmod{p}$ has a solution $t$, that is, $t^n\equiv r\pmod{p}$.

  Consider the set
  \[S:=\{x\in\Z:\, 0\leq x<p^{\frac{1}{n}}\}\]
  and let
  \[S^n:=\{(x_0,\ldots,x_{n-1}):\, x_i\in S\ (0\leq i<n)\}.\]
  Note that $S^n$ has more than $p$ elements (because $n\geq2$). Now define the function $f:S^n\to\Fbb_p$ by
  \[f(x_0,\ldots,x_{n-1})\equiv x_0+x_1t+\cdots+x_{n-1}t^{n-1}\pmod{p}.\]
  Since $\Fbb_p$ has $p$ many elements, $S^n$ has more elements than $\Fbb_p$, so by the pigeonhole principle there are two $(m_0,\ldots,m_{n-1})\neq(m'_0,\ldots,m'_{n-1})$ in $S^n$ such that $f(m_0,\ldots,m_{n-1})=f(m'_0,\ldots,m'_{n-1})$. For $0\leq i<n$ let $a_i:=m'_i-m_i$, so
  \[f(a_0,\ldots,a_{n-1})\equiv f(m'_0,\ldots,m'_{n-1})-f(m_0,\ldots,m_{n-1})\equiv 0\pmod{p},\]
  $\bar{a}:=(a_0,\ldots,a_{n-1})\neq(0,\ldots,0)$ and $-p^{\frac{1}{n}}<a_i<p^{\frac{1}{n}}$, We show that $\bar{a}$ is as desired.\smallskip

  We proceed in a similar way as in the proof of (a) first assuming that $x^n-r$ is irreducible in  $\Q[x]$. Then
  $K:=\Q^r_n=\Q(v)$ is a field extension of $\Q$ with $v=\sqrt[n]{r}$, and it is isomorphic to $\M^r_n(\Q)$ by Lemma~\ref{Frnisom}. Set $A:=M^r_n(\bar{a})$. Since this matrix is not zero, it is invertible, so $A^{-1}\in\M^r_n(\Q)$, and even more $B:=\adj(A)\in\M^r_n(\Z)$ by Lemma~\ref{matinv}. So choose $\bar{y}\in\Z^n$ such that $B=M^r_n(\bar{y})$.

  Since $K$ is $\Q[x]/(q(x))$ with $q(x):=x^n-r$, we have that $A=M^r_n(g(x)\pmod{(q(x))})$ and $B=M^r_n(h(x)\pmod{(q(x))})$ where
  \[\begin{split}
      g(x) & :=a_0+a_1x+\cdots+a_{n-1}x^{n-1},\\
      h(x) & :=y_0+y_1x+\cdots+y_{n-1}x^{n-1}.
  \end{split}\]
  Since $AB=|A|I_n$, we get that $x^n-r$ divides $g(x)h(x)-|A|$ in $\Q[x]$, and actually in $\Z[x]$ because both polynomials have coefficients in $\Z$ and $x^n-r$ is monic. Then $g(x)h(x)=j(x)q(x)+|A|$ for some $j(x)\in\Z[x]$.

  To finish the proof, note that $g(t)h(t)-|A|=(t^n-r)j(t)\equiv 0\pmod{p}$, so $g(t)h(t)\equiv|A|\pmod{p}$. On the other hand, we know that $g(t)\equiv f(a_0,\ldots,a_{n-1})\equiv 0\pmod{p}$ so $|A|\equiv 0\pmod{p}$, that is, $D^r_n(a_0,\ldots,a_{n-1})\equiv 0\pmod{p}$.\smallskip

  For the general proof of (c) we work in $F^r_n$, which is isomorphic to $\M^r_n(F)$. Again set $A:=M^r_n(\bar{a})$ which is in $M^r_n(F)$, so $B:=\adj(A)\in\M^r_n(F)$ by Lemma~\ref{matinv}. Like above, since $AB=|A|I_n$ we have two polynomials $g(x),h(x)\in F[x]$, with $g(x)$ as above, such that $x^n-r$ divides $g(x)h(x)-|A|$, so $g(x)h(x)=j(x)q(x)+|A|$ for some $j(x)\in F[x]$. Exactly as in the last part of the previous argument, we conclude that $D^r_n(\bar{a})\equiv 0\pmod{p}$.
\end{proof}


Thanks to the results in Section~\ref{sec:prep}, the previous result takes a simple form when $n$ is a prime.

\begin{corollary}\label{weakerq}
   Let $p$ and $q$ be primes. Then the equation $x^q\equiv r\pmod{p}$ has a solution iff the equation $D^r_q(x_0,\ldots,x_{q-1})\equiv 0\pmod{p}$ has a \underline{non-trivial} solution.
\end{corollary}
\begin{proof}
   The direction from left to right follows from Theorem~\ref{weaker}(c). For the converse, if the equation $x^q\equiv r\pmod{p}$ does not have a solution then the polynomial $x^q-r_0$ is irreducible in $\Fbb_p[x]$ by Corollary~\ref{irrx^q} where $r_0\in\Fbb_p$ is the residue of $r$ modulo $p$, so $D^r_q(x_0,\ldots,x_{q-1})\equiv 0\pmod{p}$ does not have a non-trivial solution by Theorem~\ref{weaker}(a).
\end{proof}


The next result is Theorem~\ref{highresmain2}, which is a weakening of (2)$\imp$(1) of Problem~\ref{problemDet=p}. This actually checks this implication when $n$ is a prime (for any $r\in\Z$).

\begin{theorem}\label{onedir}
   Assume that $p$ is a prime, $r\in\Z$, $r\equiv r_0\pmod{p}$ with $r_0\in\Fbb_p$ and $n\geq 2$. If the polynomial $x^n-r_0$ is irreducible in $\Fbb_p[x]$ then $D^r_n(x_0,\ldots,x_{n-1})=p$ does not have a solution in the integers.

   In particular, if $q$ is a prime and $x^q\equiv r\pmod{p}$ does not have a solution then $D^r_q(x_0,\ldots,x_{q-1})=p$ does not have a solution in the integers.
\end{theorem}
\begin{proof}
   By Theorem~\ref{weaker}, if $x^n-r_0$ is irreducible in $\Fbb_p[x]$ then $D^r_n(x_0,\ldots,x_{n-1})\equiv 0\pmod{p}$ does not have a \emph{non-trivial} solution. Thus, if $D^r_n(x_0,\ldots,x_{n-1})=p$ has a solution $a_0,\ldots,a_{n-1}\in\Z$, then every $a_i$ must be a multiple of $p$. But this implies that $D^r_n(a_0,\ldots,a_{n-1})$ is a multiple of $p^n$, so it cannot be equal to $p$ because $n\geq2$.
\end{proof}

We can use Theorem~\ref{weaker} to solve Problem~\ref{problemDet=p} for $n=2$, i.e., Theorem~\ref{pnorm2}. In fact, this is valid for $-1$ and $-2$ in the place of $2$, which yield well known results. 

\begin{theorem}\label{D2}
   Let $r\in\{-2,-1,2\}$.
   If $p$ is a prime then the equation $x^2\equiv r\pmod{p}$ has a solution iff the equation $D^r_2(x_0,x_1)=p$ has a solution in the integers.
\end{theorem}
\begin{proof}
   One implication follows by Theorem~\ref{onedir} because $2$ is prime. 
   So we show that, whenever $x^2\equiv r\pmod{p}$ has a solution, the equation $D^r_2(x_0,x_1)=p$ has a solution in the integers, for $r\in\{-2,-1,2\}$.

   By Corollary~\ref{weakerq}, the equation $D^r_2(x_0,x_1)\equiv 0\pmod{p}$ has a non-trivial solution $(a,b)$. Hence $p$ divides $D^r_2(a,b)=a^2-b^2 r$. According to Theorem~\ref{weaker}(c), we can find $a$ and $b$ between $-p^{\frac{1}{2}}$ and $p^{\frac{1}{2}}$.\smallskip

   \noindent\emph{Case $r=2$.}
   We claim that $-2p<a^2-2b^2<p$. Two cases: if $a^2\geq2b^2$ then $0\leq a^2-2b^2\leq a^2<p$;  if $a^2<2b^2$ then $-2p<-2b^2\leq a^2-2b^2<0$, so the claim follows.

   Now, since $-2p<D^2_2(a,b)=a^2-2b^2<p$ and $p \mid  D^2_2(a,b)$, we must have that $D^2_2(a,b)=-p$ (it can not be zero because $p$ must not divide both $a$ and $b$).

   Note that $D^2_2(1,1)=1^2-2\cdot 1^2=-1$, so
   \[p=\begin{vmatrix}
      a & 2b\\
      b & a
   \end{vmatrix}
   \cdot
   \begin{vmatrix}
      1 & 2\\
      1 & 1
   \end{vmatrix}
    =
    \begin{vmatrix}
      a+2b & 2(a+b)\\
      a+b & a+2b
    \end{vmatrix}\]
    Hence $x_0:=a+2b$ and $x_1=a+b$ form an integer solution of $D^2_2(x_0,x_1)=p$.\smallskip

    \noindent\emph{Case $r=-1$.} It is clear that $0<a^2+b^2<2p$, so $a^2+b^2=p$.\smallskip

    \noindent\emph{Case $r=-2$.} Note that $0<a^2+2b^2<3p$, so either $a^2+2b^2=p$ or $a^2+2b^2=2p$. In the first case we are done; in the second case $a$ must be even, so $a=2a_0$ for some $a_0\in\Z$, and $2p=a^2+2b^2=4a_0^2+2b^2$, hence $D^{-2}_2(b,a_0)=p$.
\end{proof}




\section{Discussions}\label{sec:disc}

Problem~\ref{problemDet=p} cannot be generalized by simply replacing $2$ by any $r\in\Z$. For $n=2$, it is known it is fine for $r\in\{-2,-1,2\}$ as shown in Theorem~\ref{D2}, but other values of $r$ are problematic. For example, $3y^2+p$ is never a square when $p\equiv 3\pmod{4}$ (because it is $3$ or $2$ modulo $4$), so $D^3_2(x,y)=p$ does not have a solution for those $p$. However, there are primes $p\equiv 3\pmod{4}$ such that $x^2\equiv 3\pmod{p}$ has a solution, for example, $p=11$. In this case, it could be conjectured that the equation $D^3_2(x,y)=p$ has a solution iff $x^2\equiv 3\pmod{p}$ has a solution \underline{and} $p\equiv 1\pmod{4}$. This motivates:

\begin{problem}
   For $n\geq 2$ (particularly $n=2$) and $r\in\Z$ (or just free of $n$-powers), what are suitable necessary and sufficient conditions for a prime $p$ to get that $D^r_n(\bar{x})=p$ has a solution in the integers?
\end{problem}

As discussed in the introduction, the solution of Problem~\ref{problemDet=p} should be related to the characterization of primes (or irreducible) elements in $\Z[\sqrt[n]{2}]$, which looks very complex for general values of $n$. In the post~\cite{MSE} it is hinted that Problem~\ref{problemDet=p} is true for $n=3$ by looking at $\Z[\sqrt[3]{2}]$ with tools that we do not deal with in this paper.

Some results of Section~\ref{sec:partial} can be generalized when $x^n-r$ is replaced by any monic polynomial in $\Z[x]$. If $R$ is an integral domain and $q(x)\in R[x]$ is a monic polynomial of degree $n>0$, the theory in the first part of Section~\ref{sec:prep} can be generalized in the context of $R_{q(x)}:=R[x]/(q(x))$:
\begin{enumerate}[(I)]
    \item $R_{q(x)}$ is a free $R$-module (and an $R$-algebra) with basis $\{1,u,\ldots,u^{n-1}\}$ where $u:=x\pmod{(q(x))}$
    \item For any $z\in R_{q(x)}$ there is a unique matrix $M_{q(x)}(z)$ that characterizes the endomorphism $R_{q(x)}\to R_{q(x)}$, $w\mapsto zw$ as in Lemma~\ref{mathomom}.
    \item Set $\M_{q(x)}:=\M_{q(x)}(R)=\{M_{q(x)}(z):\, z\in R_{q(x)}\}$. The function $M_{q(x)}:R_{q(x)}\to\M_{q(x)}$ is an $R$-algebra isomorphism.
    \item For any $z\in R_{q(x)}$ set $D_{q(x)}(z):=|M_{q(x)}(z)|$. Then, for any $z,z'\in R_{q(x)}$,
    \[D_{q(x)}(zz')=D_{q(x)}(z)D_{q(x)}(z').\]
    When $\bar{x}=(x_0,\ldots,x_{n_1})\in R$, denote $D_{q(x)}(\bar{x}):=D_{q(x)}(z)$ where $z=\sum_{i=0}^{n-1}x_i u^i\in R_{q(x)}$.
    \item If $A\in\M_{q(x)}(R)$ then $\adj(A)\in\M_{q(x)}(R)$.
\end{enumerate}

Using this theory, we obtain the following results (with similar proofs as in Section~\ref{sec:partial}).

\begin{theorem}\label{sec5gen}
Let $p$ be a prime, $q(x)\in\Z[x]$ a monic polynomial of degree $n>0$, and let $q_0(x)\in\Fbb_p[x]$ be the polynomial resulting from $q(x)$ by changing its coefficients by their residues modulo $p$. Then:
\begin{enumerate}[(1)]
    \item $q_0(x)$ is irreducible in $\Fbb_p[x]$ iff the equation $D_{q(x)}(x_0,\ldots,x_{n-1})\equiv 0\pmod{p}$ does not have a non-trivial solution in the integers.
    \item If $q(x)$ is reducible in $\Q[x]$ then the equation $D_{q(x)}(\bar{x})=0$ has a non-trivial solution in the integers.
    \item If $n\geq 2$ and the equation $q_0(x)\equiv 0\pmod{p}$ has a solution then the equation $D_{q(x)}(x_0,\ldots,x_{n-1})\equiv 0\pmod{p}$ has a non-trivial solution in the integers with $-p^{\frac{1}{n}}<x_i<p^{\frac{1}{n}}$ for any $i$.
    \item If $n\geq 2$ and $q_0(x)$ is irreducible in $\Fbb_p[x]$ then the equation $D_{q(x)}(\bar{x})=p$ does not have a solution in the integers.
\end{enumerate}
\end{theorem}

As a digression, the equation $D^2_3(x_0,x_{1},x_{2})=p$ motivates the following.

\begin{problem}\label{primes}
 Assume that $a,b,c\in\{1,2,3\}$ and that $p$ is a prime.
Does the equation $x^a+2y^b+4z^c=p$ have a solution in the integers?
\end{problem}

For any $p\in\Z$ (not necessarily prime): it is easy to find a solution when either $a$, $b$ or $c$ is equal to 1; and the case $a=b=c=2$ has a positive answer, as mentioned in~\cite[\S 13.3, Prob.~8(a)]{burton}.

So this leaves the case $2\leq\min\{a,b,c\}\leq\max\{a,b,c\}=3$. By running computations in Wolfram Mathematica with the command \texttt{FindInstance} (see below), a solution was not found for some primes in all the subcases (but this is not a proof that the solution does not exist).
\begin{center}
   \verb'FindInstance[x^a+2y^b+4z^c==p,{x,y,z},Integers]'
\end{center}
See details in Tables~\ref{tablefail} and~\ref{tablefail2}: in Table~\ref{tablefail} we look at the case when at least two of $a,b,c$ are equal to $3$, where solutions were not found for some primes below $10000$; in Table~\ref{tablefail2} we look at the case when only one of $a,b,c$ is equal to $3$, where solutions were not found for some primes beyond $20000$.

\begin{table}[ht]
  \begin{center}
    \begin{tabular}{|l||l|}
     \hline
      & Primes $p$ where a solution was not found\\
     $(a,b,c)$  & with \texttt{FindInstance} among the first\\
      & $1000$ primes\\
     \hline\hline
     $(2,3,3)$ & $2069$, $5303$, $6101$\\
     \hline
     $(3,2,3)$ & $2207$, $2383$\\
     \hline
     $(3,3,2)$ & $2039$, $2083$, $3371$, $4027$, $6143$, $6997$, $7699$\\
     \hline
     $(3,3,3)$ & $4079$, $4091$, $6449$, $7507$\\
     \hline
    \end{tabular}
    \caption{Instances among the first $1000$ primes where a solution of $x^a+2y^b+4z^c=p$ was not found in Wolfram Mathematica with the command \texttt{FindInstance}, in the case when at least two of $a,b,c$ are equal to $3$.}
    \label{tablefail}
  \end{center}
  \end{table}

 \begin{table}[ht]
  \begin{center}
    \begin{tabular}{|l||l|}
     \hline
      & First four primes $p$ where a\\
     $(a,b,c)$  & solution was not found with\\
       & \texttt{FindInstance}\\
     \hline\hline
     $(2,2,3)$ & $22691$, $25903$, $27191$, $27241$\\ 
     \hline
     $(2,3,2)$ & $37571$, $39191$, $41263$, $44357$\\
     \hline
     $(3,2,2)$ & $24907$, $51043$, $51637$, $53717$\\
     \hline
    \end{tabular}
    \caption{First four prime $p$ instances where a solution of $x^a+2y^b+4z^c=p$ was not found in Wolfram Mathematica with the command \texttt{FindInstance}, in the case when only one of $a,b,c$ is equal to $3$.}
    \label{tablefail2}
  \end{center}
  \end{table}


\end{document}